\documentclass[12pt]{amsart}
\usepackage{amssymb,latexsym,esint}
\usepackage{enumerate}
\usepackage{float}

\usepackage{mathtools}
\usepackage{amsfonts}
\usepackage{amssymb}
\usepackage{amsmath}
\usepackage{graphicx}
\usepackage{graphicx,color}
\usepackage[normalem]{ulem}
\usepackage{todonotes}

\numberwithin{equation}{section}
\makeatletter
\@namedef{subjclassname@2020}{%
  \textup{2020} Mathematics Subject Classification}
\makeatother

\usepackage{amsmath,amstext,amsthm,amsxtra,mathtools,amssymb}

\usepackage[bookmarksopen,bookmarksdepth=3,colorlinks,citecolor=red,pagebackref,hypertexnames=true]{hyperref}
\usepackage[backrefs,msc-links,nobysame,initials,non-sorted-cites]{amsrefs}

\usepackage[normalem]{ulem}
\usepackage{lmodern}

\usepackage{lmodern}

\usepackage[backrefs,msc-links,nobysame,initials]{amsrefs}

\allowdisplaybreaks
\newtheorem{prop}{Proposition}
\newtheorem{thm}{Theorem}

\newtheorem{lem}{Lemma}

\theoremstyle{definition}

\theoremstyle{remark}

\theoremstyle{definition}
\newtheorem{conjecture}{Conjecture}

\begin{document}

\title[Sharp Weak Type Estimates for a Family of Soria Bases]{Sharp Weak Type Estimates for a family of Soria bases}

\author{Dmitry Dmitrishin}
\address{D. D.: Department of Applied Mathematics, Odessa National Polytechnic University, Odessa 65044, Ukraine}
\email{\href{mailto: dmitrishin@opu.ua}{dmitrishin@opu.ua}}

\author{Paul Hagelstein}
\address{P. H.: Department of Mathematics, Baylor University, Waco, Texas 76798}
\email{\href{mailto:paul_hagelstein@baylor.edu}{paul\_hagelstein@baylor.edu}}
\thanks{P. H. is partially supported by a grant from the Simons Foundation (\#521719 to Paul Hagelstein).}

\author{Alex Stokolos}
\address{A. S.: Department of Mathematical Sciences, Georgia Southern University, Statesboro, Georgia 30460}
\email{\href{mailto:astokolos@GeorgiaSouthern.edu}{astokolos@GeorgiaSouthern.edu}}

\subjclass[2020]{Primary 42B25}
\keywords{maximal functions, differentiation basis}

%%%%%%%%%%%%%%%%%%%%%%%%%%%%%% SECTION SECTION SECTION
\begin{abstract}
Let $\mathcal{B}$ be a collection of rectangular parallelepipeds in $\mathbb{R}^3$ whose sides are parallel to the coordinate axes and such that $\mathcal{B}$ contains parallelepipeds with side lengths of the form $s, \frac{2^N}{s} , t $, where $s, t > 0$ and $N$ lies in a nonempty subset $S$  of the natural numbers.    We show that if $S$ is an infinite set, then the associated geometric maximal operator $M_\mathcal{B}$ satisfies the weak type estimate

$$\left|\left\{x \in \mathbb{R}^3 : M_{\mathcal{B}}f(x) > \alpha\right\}\right| \leq C \int_{\mathbb{R}^3} \frac{|f|}{\alpha} \left(1 + \log^+ \frac{|f|}{\alpha}\right)^{2}$$

\noindent but does not satisfy an estimate of the form

$$\left|\left\{x \in \mathbb{R}^3 : M_{\mathcal{B}}f(x) > \alpha\right\}\right| \leq C \int_{\mathbb{R}^3} \phi\left(\frac{|f|}{\alpha}\right)$$
for any convex increasing function $\phi: \mathbb[0, \infty) \rightarrow [0, \infty)$ satisfying the condition
$$\lim_{x \rightarrow \infty}\frac{\phi(x)}{x (\log(1 + x))^2} = 0\;.$$

\end{abstract}

\maketitle
 
\section{Introduction}

This paper is concerned with sharp weak type estimates for a class of maximal operators naturally arising from work surrounding the so-called Zygmund conjecture in multiparameter harmonic analysis.  Let us recall that the \emph{strong maximal operator} $M$ is defined on $L_{\textup{loc}}^1(\mathbb{R}^n)$ by
$$Mf(x) = \sup_{x \in R} \frac{1}{|R|}\int_R |f|\;,$$
where the supremeum is over all rectangular parallelepipeds in $\mathbb{R}^n$ containing $x$ whose sides are parallel to the coordinate axes.  
 An important inequality associated to the strong maximal operator is
$$\left|\left\{x \in \mathbb{R}^n : Mf(x) > \alpha\right\}\right| \leq C_n \int_{\mathbb{R}^n} \frac{|f|}{\alpha} \left(1 + \log^+ \frac{|f|}{\alpha}\right)^{n-1}\;.$$   This inequality may be found in de Guzm\'an \cite{guzman1974, guzman} (see also the related paper \cite{cf1975} of A. C\'ordoba and R. Fefferman  as well as the paper \cite{favacapri} of Capri and Fava) and may be used to provide a proof of the classical  \emph{Jessen-Marcinkiewicz-Zygmund Theorem} \cite{jmz}, which tells us that the integral of any function in $L (\log^+L)^{n-1}(\mathbb{R}^n)$ is strongly differentiable.

Now, the strong maximal operator in $\mathbb{R}^n$ is associated to an $n$-parameter basis of rectangular parallelepipeds. It is natural to consider  weak type estimates for maximal operators in $\mathbb{R}^n$ associated to $k$-parameter bases.  The \emph{Zygmund Conjecture} in this regard is the following:

\begin{conjecture}[Zygmund Conjecture; now disproven]
Let $\mathcal{B}$ be a collection of rectangular parallelepipeds in $\mathbb{R}^n$ whose sides are parallel to the coordinate axes and whose sidelengths are of the form $$\phi_1(t_1, \ldots, t_k), \ldots, \phi_n(t_1, \ldots, t_k)$$ where the functions $\phi_i$ are nonnegative and increasing in each variable separately.  Define the associated maximal operator $M_\mathcal{B}$ by
$$M_\mathcal{B}f(x) = \sup_{x \in R \in \mathcal{B}} \frac{1}{|R|}\int_R |f|\;.$$
Then $M_\mathcal{B}$ satisfies the weak type estimate
\begin{equation}\label{e1}\left|\left\{x \in \mathbb{R}^n : M_\mathcal{B} f(x) > \alpha\right\}\right| \leq C_n \int_{\mathbb{R}^n} \frac{|f|}{\alpha} \left(1 + \log^+ \frac{|f|}{\alpha}\right)^{k-1}\;.\end{equation}
\end{conjecture}
This conjecture was disproven by Soria in \cite{soria}.  That being said, it does hold in many important cases.  For example, A. C\'ordoba proved in \cite{cordoba} that the Zygmund Conjecture holds in the case that  $\mathcal{B}$ consists of rectangular parallelepipeds in $\mathbb{R}^3$ with sides parallel to the coordinate axes and whose sidelengths are of the form $s, t, \phi(s,t)$, where $\phi$ is nonnegative and increasing in the variables $s, t$ separately.  Of particular interest to us in this paper is the following extension of C\'ordoba's result due to Soria in \cite{soria}:

\begin{prop}\label{prop1}
Let $\mathcal{B}$ be a collection of rectangular parallelepipeds in $\mathbb{R}^3$ whose sides are parallel to the coordinate axes.  Furthermore, suppose that, given a parallelepided $R$ in $\mathcal{B}$ of sidelengths $r_1, r_2, r_3$ and another parallelepided $R'$ in $\mathcal{B}$ of sidelengths $r_1', r_2', r_3'$, if $r_1 > r_1'$, then either $r_2 > r_2'$ or $r_3 > r_3'$.  Then

$$\left|\left\{x \in \mathbb{R}^3 : M_\mathcal{B} f(x) > \alpha\right\}\right| \leq C \int_{\mathbb{R}^3} \frac{|f|}{\alpha} \left(1 + \log^+ \frac{|f|}{\alpha}\right)\;.$$
\end{prop}

Note that this proposition encompasses bases that can be quite different in character than the ones consider by C\'ordoba.   In particular, in \cite{soria} Soria mentions as an example the basis of parallelepipeds with sidelengths of the form $s, t, \frac{1}{t}$.
\\

At this point we introduce another strand of research associated to Zygmund's Conjecture.  It is natural to consider, given a translation invariant basis $\mathcal{B}$ of rectangular parallelepipeds, whether or not the \emph{sharp} weak type estimate associated to $M_\mathcal{B}$ must be of the form
$$\left|\left\{x \in \mathbb{R}^n : M_\mathcal{B} f(x) > \alpha\right\}\right| \leq C_n \int_{\mathbb{R}^n} \frac{|f|}{\alpha} \left(1 + \log^+ \frac{|f|}{\alpha}\right)^{k-1}\;$$ for some \emph{integer} $1 \leq k \leq n$.  In \cite{stokolos1988}, Stokolos proved the
 following:

\begin{prop}\label{prop2}
Let $\mathcal{B}$ be a translation invariant basis of rectangles in $\mathbb{R}^2$ whose sides are parallel to the coordinate axes.  If $\mathcal{B}$ does not satisfy the weak type $(1,1)$ estimate
$$|\{x \in \mathbb{R}^2 : M_\mathcal{B} f(x) > \alpha\}| \leq C \int_{\mathbb{R}^2} \frac{|f|}{\alpha}$$
then $M_\mathcal{B}$ satisfies the weak type estimate
$$\left|\left\{x \in \mathbb{R}^2 : M_\mathcal{B} f(x) > \alpha\right\}\right| \leq C \int_{\mathbb{R}^2} \frac{|f|}{\alpha} \left(1 + \log^+  \frac{|f|}{\alpha}\right)\;$$
but does \emph{not} satisfy a weak type estimate of the form 
$$|\{x \in \mathbb{R}^2 : M_\mathcal{B} f(x) > \alpha\}| \leq C \int_{\mathbb{R}^2} \phi\left(\frac{|f|}{\alpha}\right)$$
for any nonnegative convex increasing function $\phi$ such that $\phi(x) = o(x \log x)$ as x tends to infinity.
\end{prop}
In essence, this proposition tells us that, if $\mathcal{B}$ is a translation invariant basis of rectangles in $\mathbb{R}^2$ whose sides are parallel to the coordinate axes, then the optimal weak type estimate for $M_\mathcal{B}$ must be inequality \ref{e1} for $k = 1$ or $k = 2$. Optimal weak type estimates of this form when, say, $k = \frac{3}{2}$ are ruled out.  The proof of Stokolos' result is very delicate and involves the idea of \emph{crystallization} that we will return to.

It is of interest that Proposition \ref{prop2} has at the present time never been extended to encompass translation invariant bases consisting of (some, but not all) rectangular parallelepipeds in dimensions 3 or higher.  In particular, one might expect that the optimal weak type estimate for the maximal operator associated to such a  basis of parallelepipeds in $\mathbb{R}^3$ would be inequality \ref{e1} when $n = 3$ and $k$ is either $1$, $2$, or $3$.
\\

The purpose of this paper is, motivated by Propositions \ref{prop1} and \ref{prop2} above, to consider sharp weak type estimates associated to the translation invariant basis of rectangular parallelepipeds in $\mathbb{R}^3$ whose sides are parallel to the coordinate axes and whose sidelengths are of the form $s, \frac{2^N}{s} , t $, where $s, t > 0$ and $N$ lies in a nonempty subset $S$ of the natural numbers.  The end result, although not its proof, is strikingly straightforward and is stated as follows:

\begin{thm}\label{t1}
Let $\mathcal{B}$ be a collection of rectangular parallelepipeds in $\mathbb{R}^3$ whose sides are parallel to the coordinate axes and such that $\mathcal{B}$ contains all parallelepipeds with side lengths of the form $s, \frac{2^N}{s} , t $, where $s, t > 0$ and N lies in a nonempty subset $S$  of the natural numbers. 

If $S$ is a finite set, then the associated geometric maximal operator $M_\mathcal{B}$ satisfies the weak type estimate of the form

\begin{equation} \label{e2}\left|\left\{x \in \mathbb{R}^3 : M_{\mathcal{B}}f(x) > \alpha\right\}\right| \leq C \int_{\mathbb{R}^3} \frac{|f|}{\alpha}\left(1 + \log^+ \frac{|f|}{\alpha}\right)\;\end{equation}
but does not satisfy an estimate of the form

$$\left|\left\{x \in \mathbb{R}^3 : M_{\mathcal{B}}f(x) > \alpha\right\}\right| \leq C \int_{\mathbb{R}^3} \phi\left(\frac{|f|}{\alpha}\right)$$
for any convex increasing function $\phi: \mathbb[0, \infty) \rightarrow [0, \infty)$ satisfying the condition
$$\lim_{x \rightarrow \infty}\frac{\phi(x)}{x (\log(1 + x))} = 0\;.$$
\\

If $S$ is an infinite set, then the associated geometric maximal operator $M_\mathcal{B}$ satisfies a weak type estimate of the form

$$\left|\left\{x \in \mathbb{R}^3 : M_{\mathcal{B}}f(x) > \alpha\right\}\right| \leq C \int_{\mathbb{R}^3} \frac{|f|}{\alpha} \left(1 + \log^+ \frac{|f|}{\alpha}\right)^{2}$$

\noindent but does not satisfy an estimate of the form

$$\left|\left\{x \in \mathbb{R}^3 : M_{\mathcal{B}}f(x) > \alpha\right\}\right| \leq C \int_{\mathbb{R}^3} \phi\left(\frac{|f|}{\alpha}\right)$$
for any convex increasing function $\phi: \mathbb[0, \infty) \rightarrow [0, \infty)$ satisfying the condition
$$\lim_{x \rightarrow \infty}\frac{\phi(x)}{x (\log(1 + x))^2} = 0\;.$$

\end{thm}

The remainder of the paper is devoted to a proof of this theorem.  Note that for inequality \ref{e2}, it is easily seen that the constant $C$ is at most linearly dependent on the number of elements in $S$, although the sharp dependence of $C$ on the number of elements of $S$ is potentially a quite difficult issue that we do not treat here.  The primary content of the above theorem is the sharpness of the weak type estimate of $M_\mathcal{B}$ in the case that $S$ is infinite.  In harmonic analysis we typically  show that an optimal  weak type estimate on a maximal operator is sharp by testing the operator on a bump function or the characteristic function of a small interval or rectangular parallelepiped.  This can be done, for instance, with the Hardy-Littlewood maximal operator, the strong maximal operator, or even the maximal operator associated to rectangles whose sides are parallel to the axes with sidelengths of the form $t, \frac{1}{t}$ \cite{soria}.  However, in dealing with maximal operators associated to rare bases of the type featured in Theorem \ref{t1}, such simple functions \emph{do not} provide  examples illustrating the sharpness of the optimal weak type results, and more delicate constructions such as will be seen here are needed.
\\

We remark that a recent paper of D'Aniello and Moonens \cite{dm2017} also treats the subject of translation invariant rare bases; in particular they provide  sufficient conditions on a rare basis $\mathcal{B}$ for the estimate \ref{e1} to be sharp when $k=n$.   However, certain bases covered in Theorem \ref{t1} (such as when $S = \{2^{m^m}: m \in \mathbb{N}\}$) do not fall into the scope of those considered in their paper, although the interested reader is strongly encouraged to consult it. 
\\

{\bf{Acknowledgment:}}  We wish to thank Ioannis Parissis as well as the referees for their helpful comments and suggestions regarding this paper.

\section{Crystallization and Preliminary Weak Type Estimates}

In this section, we shall introduce a collection of two-dimensional ``crystals'' that we will use to prove Theorem \ref{t1}.  We remark that similar types of crystalline structures were used by Stokolos in \cite{stokolos1988, stokolos2005, stokolos2006} as well as by Hagelstein and Stokolos in \cite{hs2011}.
\\

Let $m_1 < m_2 < \cdots$ be an increasing sequence of natural numbers.  We may associate to this sequence and any $k \in \mathbb{N}$ a set in $[0, 2^{m_k}]$ denoted by $Y_{\left\{m_j\right\}_{j=1}^k}$  defined by

$$Y_{\left\{m_j\right\}_{j=1}^k} = \left\{t \in [0, 2^{m_k}]:   \sum_{j=1}^k r_{0}\left(\frac{t}{2^{m_{j}} }\right) = k\right\}\;.$$ 
Here $r_0(t)$ denotes the standard Rademacher function defined on $[0,1)$ by

$$r_0(t) = \chi_{[0,\frac{1}{2}]}(t) - \chi_{(\frac{1}{2}, 1)}(t)\;$$ and extended to be $1$-periodic on $\mathbb{R}$.

Note that
$$\mu_1(Y_{\left\{m_j\right\}_{j=1}^k}) = 2^{-k} 2^{m_k}\;.$$ 
Associated to the set $Y_{\left\{m_j\right\}_{j=1}^k}$ is the \emph{crystal} $Q_{\left\{m_j\right\}_{j=1}^k} \subset [0, 2^{m_k}] \times [0, 2^{m_k}]$ defined by
$$Q_{\left\{m_j\right\}_{j=1}^k} = Y_{\left\{m_j\right\}_{j=1}^k} \times Y_{\left\{m_j\right\}_{j=1}^k}\;.$$  Note
$$\mu_2(Q_{\left\{m_j\right\}_{j=1}^k}) = 2^{-2k} 2^{2m_k}\;.$$
Here $\mu_j$ refers to the Lebesgue measure on $\mathbb{R}^j$.  

We also associate to $\{m_j\}_{j=1}^k$ the geometric maximal operator $M_{\left\{m_j\right\}_{j=1}^k}$ defined on $L_{loc}^1(\mathbb{R}^2)$ by
$$M_{\left\{m_j\right\}_{j=1}^k} f(x) = \sup_{x \in R}\frac{1}{|R|}\int_R |f|\;,$$
where the supremum is over all rectangles in $\mathbb{R}^2$ containing $x$ whose sides are parallel to the coordinate axes with areas in the set $\{2^{m_1}, \ldots, 2^{m_k}\}$.

 In the case that the context is clear, we may refer to the set $Y_{\left\{m_j\right\}_{j=1}^k}$ simply as $Y_k$, the set $Q_{\left\{m_j\right\}_{j=1}^k}$ simply as $Q_k$, and the maximal operator $M_{\left\{m_j\right\}_{j=1}^k}$ simply as $M_k$.

A few basic observations regarding the sets $Y_k$ and $Q_k$ are in order.

First, note that $Y_{k + 1}$ is a disjoint union of $\frac{2^{m_{k+1} - 1}}{2^{m_k}}$ copies of $Y_{k}$.  In fact, defining the translation $\tau_s E$ of a set $E$ in $\mathbb{R}$ by
$\chi_{\tau_s E}(x) = \chi_E (x - s)$, we have
$$Y_{k + 1} = \bigcup_{l = 0}^{\frac{2^{m_{k+1} - 1}}{2^{m_k}} -1} \tau_{l 2^{m_k}} Y_{k}\;.$$   Furthermore, by induction we see that if $1 \leq r \leq k$ we have $Y_{k + 1}$ is a disjoint union of

$$\frac{2^{m_{k+1} - 1}}{2^{m_k}} \cdot \frac{2^{m_{k} - 1}}{2^{m_{k-1}}}\cdots \frac{2^{m_{r+1} - 1}}{2^{m_{r}}} = 2^{m_{k+1} - m_r - k + r - 1} $$
 copies of $Y_{r}$\;, with
 
  $$Y_{k + 1} = \bigcup_{(l_r, \ldots, l_k) \atop 0 \leq l_i \leq 2^{m_{i+1} - m_i - 1} -1}\tau_{l_{r}2^{m_r}}\tau_{l_{r+1}2^{m_{r+1}}}\cdots \tau_{l_k 2^{m_{k}}}Y_{r} \;.$$
  
  We also remark that the average of $\chi_{Y_{k}}$ over $[0, 2^{m_j}]$ is exactly $2^{-j}$, and moreover the average of $\chi_{Y_{k}}$ over any translate   $\tau_{l_{j}2^{m_j}}\tau_{l_{j+1}2^{m_{j+1}}}\cdots \tau_{l_{k-1} 2^{m_{k-1}}}[0, 2^{m_j}]$ with \mbox{ $0 \leq l_i \leq 2^{m_{i+1} - m_i - 1} -1$} is also $2^{-j}$\;.  Observe that the number of such translates is
 $$2^{m_{j+1} - m_j - 1} \cdot 2^{m_{j+2} - m_{j+1} - 1} \cdots 2^{m_{k} - m_{k-1} - 1} = \; 2^{m_k - m_j +j - k}.$$

We now consider how $M_{k}$ acts on $\chi_{Q_{k}}$.  We will do so in the special case that, for $1 \leq j \leq \frac{k}{2}$ we have that $m_{k-j} \leq m_{k - j + 1} - m_j$.  (This will  be the case if the $m_j$ increase rapidly in $j$, for example if $m_{j+1} \geq 2 m_j$ for all $j$.)  

Fix now $1 \leq j \leq \frac{k}{4}$.  We are going to show that there exist
$$2^{m_k - m_{k - j + 1} + m_j - j} \cdot 2^{m_k - m_j - k + j}\;=\;2^{2m_k - m_{k-j+1} - k}$$ pairwise a.e. disjoint rectangles with sides parallel to the coordinate axes in $[0, 2^{m_k}] \times [0, 2^{m_k}]$ whose areas are all $2^{m_{k-j+1}}$ and such that the average of $\chi_{Q_{k}}$ over each of these rectangles is  $2^{-k}$.  Moreover, each of these rectangles will be a translate of $[0, 2^{m_j}] \times [0, 2^{m_{k - j + 1} - m_j}]$.   Accordingly, the measure of the union of these rectangles will be $2^{2m_k - k}$.

We have already indicated above that the average of $\chi_{Y_{k}}$
over each of  $2^{m_k - m_j -k + j}$  pairwise a.e. disjoint translates of $[0, 2^{m_j}]$ is $2^{-j}$.  Somewhat more technically, we now need to prove that the average of $\chi_{Y_{k}}$ over $2^{m_k - m_{k - j + 1} + m_j - j}$ pairwise a.e. disjoint intervals of length $2^{m_{k - j + 1} - m_j}$ is  equal to $2^{j-k}$.

Note that the average of $\chi_{Y_{k}}$ over $[0, 2^{m_{k-j}}]$ is $2^{j-k}$ as well as any  translate $\tau[0, 2^{m_{k-j}}]$ of this interval where $\tau$ is of the form $l \cdot 2^{m_{k-j}}$ for $0 \leq l \leq 2^{m_{k - j + 1} - m_j - m_{k-j} } - 1$\;.   The union of these intervals is the interval $I := [0, 2^{m_{k - j + 1} - m_j}]$ over which the average of $\chi_{Y_k}$ is $2^{j-k}$.   It is especially important to recognize here that $$Y_{k} \cap [0, 2^{m_{k - j + 1} - 1}] = \bigcup_{i = 0}^{2^{m_{k - j + 1} - m_{k - j} - 1}}
 \tau_{i2^{m_{k-j}}}Y_{k-j}\;,$$ where the latter is a pairwise a.e. disjoint union. It is here that we need the condition that $m_{k - j} \leq m_{k - j + 1} - m_j$, so that $[0, 2^{m_{k - j + 1} - m_j}]$ 
 can be tiled by pairwise a.e. disjoint intervals of length $2^{m_{k-j}}$ over which the average of $\chi_{Y_{k - j}}$ is $2^{j - k}$.
 
 Now, $[0, 2^{m_{k}}]$ contains many pairwise a.e. disjoint translates of $I \cap Y_{k}$, each of whom being contained in a collection of translates of $I$ that are themselves pairwise a.e. disjoint; we count them here.  The number of translates is the number of pairwise a.e. disjoint translates of $I$ whose union is the left half of $[0, 2^{m_{k - j + 1}}]$ (which is $2^{m_{k - j + 1} -1 - m_{k-j+1} + m_j} = 2^{m_j - 1}$) times the number of translates of $Y_{{k - j + 1}}$ needed to form $Y_{k}$ (which is $2^{m_k - m_{k - j + 1} - k + (k - j + 1)} = 2^{m_k - m_{k - j + 1} - j + 1}$.)   Hence the total number of translates is
 $$2^{ m_j - 1} \cdot  2^{m_k - m_{k - j + 1} - j + 1} = 2^{m_j + m_k - m_{k - j + 1} - j}\;.$$

 Hence, $Y_{k}$ contains $2^{m_j + m_k - m_{k - j + 1} - j}$ pairwise a.e. disjoint intervals of length $2^{m_{k - j + 1} - m_j}$ over each of which the average of $\chi_{Y_{k}}$ is $2^{j-k}$.
 As we have already shown that the average of $\chi_{Y_{k}}$
over each of  $2^{m_k - m_j -k + j}$ pairwise a.e. disjoint translates of $[0, 2^{m_j}]$ is $2^{-j}$, we have then that there exist $ 2^{m_j + m_k - m_{k - j + 1} - j} \cdot 2^{m_k - m_j -k + j} = 2^{2m_k - m_{k-j+1} - k}$ pairwise a.e. disjoint rectangles in $[0, 2^{m_k}] \times [0, 2^{m_k}]$ of size $2^{m_{k - j + 1} - m_j}\cdot 2^{m_j} = 2^{m_{k - j + 1}}$     over each of which the average of $\chi_{Q_{k}}$ is 
$2^{-j} \cdot 2^{j-k} = 2^{-k}$.   Note the measure of the union of these rectangles is
$$2^{2m_k - m_{k-j+1} - k}  \cdot 2^{m_{k - j + 1}} = 2^{2m_k - k}\;.$$
 
We come now to a crucial observation.  By the construction of $Y_{k}$, any dyadic interval of length $2^{m_j}$ is at most only half filled by the translates of intervals of length $2^{m_{j-1}}$ such that the union of those translates acting on $Y_{{j-1}}$ is $Y_j\;.$  Accordingly, the union of the above  $2^{2m_k - m_{k-j+1} - k}$ pairwise a.e. disjoint rectangles in $[0, 2^{m_k}] \times [0, 2^{m_k}]$  of size $2^{m_{k - j + 1}}$ over each of which the average of $\chi_{Q_k}$ is $2^{-k}$ is at most only half filled by the corresponding set of rectangles of size $2^{m_{k - (j-1) + 1}}$.  Hence the union of all the rectangles $R$ in $[0, 2^{m_k}] \times [0, 2^{m_k}]$ whose sides are parallel to the coordinate axes and of area in the set $\{2^{m_{k - j}} : j = 1, \ldots, \lceil \frac{k}{4}\rceil\}$ and such that the average of $\chi_{Q_k}$ over $R$ is greater than or equal to $2^{-k}$ must exceed
$\frac{1}{2} \cdot \frac{k}{4} \cdot 2^{2m_k - k} = \frac{k}{8}2^{2m_k - k}\;.$
 \\
 
 This series of observations leads to the proof of the following lemma.
 
 \begin{lem}\label{l1}
 Let the geometric maximal operator $M_{\left\{m_j\right\}_{j=1}^k}$ and the set $ Q_{\left\{m_j\right\}_{j=1}^k}$ be defined as above. Suppose for $1 \leq j \leq \frac{k}{2}$ we have that $m_{k-j} \leq m_{k - j + 1} - m_j$.  Then
 
 $$\mu_2\left(\left\{x \in [0, 2^{m_k}]\times[0, 2^{m_k}] : M_{\left\{m_j\right\}_{j=1}^k} \chi_{Q_{\left\{m_j\right\}_{j=1}^k}}(x) \geq 2^{-k}\right\}\right) \geq \frac{k}{8}2^{2m_k - k} = \frac{1}{8}\frac{k}{2^{-k}}\mu_2\left(Q_{\left\{m_j\right\}_{j=1}^k}\right)\;.$$

 \end{lem}

\section{Proof of Theorem \ref{t1}}

\begin{proof}[Proof of Theorem \ref{t1}]
Let $\mathcal{B}$ be a collection of rectangular parallelepipeds in $\mathbb{R}^3$ whose sides are parallel to the coordinate axes and such that $\mathcal{B}$ contains parallelepipeds with side lengths of the form $s, \frac{2^N}{s} , t $, where $t > 0$ and $S$ is a nonempty set consisting of natural numbers.

If $S$ is a finite set, then the associated geometric maximal operator $M_\mathcal{B}$ is comparable to the maximal operator averaging over rectangular parallelepipeds with side lengths of the form $s, \frac{1}{s}, t$.  In \cite{soria}, Soria showed that this operator maps $L(1 + \log^+ L)(\mathbb{R}^3)$ continuously into weak $L^1(\mathbb{R}^3)$ but does not map any larger Orlicz class into weak $L^1(\mathbb{R}^3)$.  So Theorem \ref{t1} holds in this case.

Suppose now $S$ is an infinite set.   Note that the maximal operator $M_\mathcal{B}$ is dominated by the strong maximal operator in $\mathbb{R}^3$, so the weak type estimate 

$$\left|\left\{x \in \mathbb{R}^3 : M_{\mathcal{B}}f(x) > \alpha\right\}\right| \leq C \int_{\mathbb{R}^3} \frac{|f|}{\alpha} \left(1 + \log^+\frac{|f|}{\alpha}\right)^{2}$$
automatically holds.

 Since $S$ is an infinite set, there exists a subset $\{m_j\}_{j = 1}^\infty$ of $S$ satisfying the condition that
$2m_j \leq m_{j+1}$ for all $j$. So the hypothesis of Lemma \ref{l1} holds for $\{m_j\}_{j=1}^k$ for all $k$.  

For each natural number $k$, we let $Z_k \subset [0, 2^{m_k}] \times [0, 2^{m_k}] \times [0, 2^{k}]$ be defined by 
$$Z_k = Q_{k} \times [0,1]\;.$$

To show the estimate 

$$\left|\left\{x \in \mathbb{R}^3 : M_{\mathcal{B}}f(x) > \alpha\right\}\right| \leq C \int_{\mathbb{R}^3} \phi\left(\frac{|f|}{\alpha}\right)$$ does \emph{not} hold 
for any convex increasing function $\phi: \mathbb[0, \infty) \rightarrow [0, \infty)$ satisfying the condition
$$\lim_{x \rightarrow \infty}\frac{\phi(x)}{x (\log(1 + x))^2} = 0\;,$$
it suffices to show that

$$\mu_3\left(\left\{x \in [0, 2^{m_k}]\times[0,2^{m_k}]\times[0, 2^k]: M_{\mathcal{B}}\chi_{Z_k}(x) \geq 2^{-k}\right\}\right) \geq \frac{1}{32} \frac{k^2}{2^{-k}}\mu_3(Z_k)\;. $$
\\

Fix $1 \leq r \leq k$.  Note that, just as 
$Y_{k}$ is a disjoint union of
$2^{m_{k} - m_r - k + r} $
 copies of $Y_{r}$\;, we have that 
 $Q_{k}$ is a disjoint union of
$2^{2(m_{k} - m_r - k + r)} $
 copies of $Q_{r}$, with each of these copies being contained in pairwise a.e. disjoint squares of sidelength $2^{m_r}$.  By Lemma \ref{l1}, for each one of these squares $\tilde{Q}$,
  $$\mu_2\left(\left\{x \in \tilde{Q}: M_{r} \chi_{\tilde{Q} \cap Q_{k}}(x)  \geq 2^{-r}\right\}\right) \geq \frac{r}{8}2^{2m_r - r}\;.$$  Note each of the rectangles associated to the maximal operator $M_{r}$ has sidelength in the set $\{2^{m_1}, \ldots, 2^{m_r}\} \subset \left\{2^{m_1}, \ldots, 2^{m_k}\right\}$ and hence for any of these rectangles $R$ the associated parallelepiped $R \times [0, 2^{k - r}]$ lies in the basis $\mathcal{B}$.  Note that if 
  $$\frac{1}{\mu_2(R)} \int_R \chi_{\tilde{Q} \cap Q_{k}} \geq 2^{-r}\;,$$
 then
 $$\frac{1}{\mu_3(R \times [0, 2^{k - r}])} \int_{R \times [0, 2^{k - r}]}\chi_{Q_{k} \times [0,1]} \geq 2^{-r}2^{r - k} = 2^{-k}\;.$$
Taking into account only the top half of these parallelepipeds, for any one of the above squares $\tilde{Q}$ we obtain
$$\mu_3\left(\left\{x \in [0, 2^{m_k}]\times[0, 2^{m_k}]\times[2^{k-r-1}, 2^{k-r}] : M_\mathcal{B}\chi_{Z_k}(x) \geq 2^{-k}\right\}\right) \geq $$

$$2^{2(m_{k} - m_r - k + r)} \mu_2\left(\left\{x \in \tilde{Q} : M_{r} \chi_{\tilde{Q} \cap Q_{k}}(x)  \geq 2^{-r}\right\}\right) \cdot 2^{k-r-1}$$

$$\geq 2^{2(m_{k} - m_r - k + r)} \frac{r}{8}2^{2m_r - r} \cdot 2^{k-r-1}\; = \; \frac{r}{16} 2^{2m_k - k}\;.$$

We now take advantage of the fact that, for different values of $r$, the sets \newline \mbox{ $[0, 2^{m_k}] \times[0, 2^{m_k}]\times [2^{k-r-1}, 2^{k-r}] $ }are pairwise a.e. disjoint.  In particular, we have

\begin{align}
&\mu_3\left(\left\{x \in [0, 2^{m_k}]\times[0, 2^{m_k}]\times[0, 2^k]: M_{\mathcal{B}}\chi_{Z_k}(x) \geq 2^{-k}\right\}\right)  \notag \\ & \geq \sum_{r = 1}^k \mu_3\left(\left\{x \in [0, 2^{m_k}]\times[0, 2^{m_k}]\times[2^{k - r - 1}, 2^{k-r}]: M_{\mathcal{B}}\chi_{Z_k}(x) \geq 2^{-k}\right\}\right)\notag \\
& \;\geq  \sum_{r = 1}^k  \frac{r}{16} 2^{2m_k - k} \geq \frac{1}{32} \frac{k^2}{2^k} 2^{2m_k} = \frac{1}{32}\frac{k^2}{2^{-k}}\mu_3(Z_k)\;, \notag
\end{align} as desired.

\end{proof}

\end{document}